\documentclass[12pt,a4paper]{amsart}

\usepackage{hyperref}
\usepackage{paralist, mathdots} %

\usepackage{amssymb} %
\usepackage{amsmath,tikz,wasysym,multirow,enumitem,lscape,multirow,subfigure}
\usepackage[section]{algorithm} %

\setlist[enumerate,1]{label = (\alph*),ref = (\alph*)}

\usepackage{mwe}
\usepackage{hyperref} 
\usepackage{amsthm}
\usepackage{mathtools}
\usepackage{soul, arydshln}
\usepackage[noadjust]{cite}\usepackage{MnSymbol}

\tikzset{every node/.style={draw, circle, inner sep=2pt},
every picture/.append style={thick,scale=0.8},
every label/.style={draw=none, rectangle}}

\newtheorem{proposition}{Proposition}[section]

\newtheorem{theorem}[proposition]{Theorem}

\newtheorem{corollary}[proposition]{Corollary}
\newtheorem{lemma}[proposition]{Lemma}
\theoremstyle{definition}
\newtheorem{example}[proposition]{Example}
\newtheorem{remark}[proposition]{Remark}

\newcommand{\npmatrix}[1]{\left( \begin{matrix} #1 \end{matrix} \right)}
\newcommand{\bigabs}[1]{\left|\,\strut#1 \,\right|}
\newcommand{\R}{\mathbb{R}}

\newcommand{\bR}{\mathbb{R}}
\newcommand{\bN}{\mathbb{N}}

\newcommand{\mc}{\mathcal}

\DeclareMathOperator{\diag}{diag}
\DeclareMathOperator{\diam}{diam}

\newcommand{\1}{{\bf 1}}
\newcommand{\0}{{\bf 0}}
\newcommand{\e}{{\bf e}}

\newcommand{\ww}{{\bf w}}
\newcommand{\xx}{{\bf x}}
\newcommand{\yy}{{\bf y}}

\let\epsilon\varepsilon

\newcommand{\tuple}{\mathbf}
\newcommand{\m}{\tuple{m}}

\newcommand{\vv}{\tuple{v}}
\newcommand{\vu}{\tuple{u}}

\newcommand{\trans}{^\top}

\newcommand{\norm}[1]{\ensuremath{\left\|{#1}\right\|}}

\newcommand{\El}{\mathcal{E}_\Lambda}

    \definecolor{helena}{rgb}{.2,.8,.4}
    \definecolor{polona}{rgb}{.2,.2,.8}
    \definecolor{rupert}{rgb}{0,.5,.5}
   \definecolor{todo}{rgb}{.8,.2,.2}

\AtBeginDocument{%
   \def\MR#1{}
}

\begin{document}
\title[Paths are  generically realisable]
{Paths are  generically realisable}
\author{Rupert H. Levene, Polona Oblak, Helena \v Smigoc}
\address[R.~H.~Levene and H.~\v Smigoc]{School of Mathematics and Statistics, University College Dublin, Belfield, Dublin 4, Ireland}
\email{rupert.levene@ucd.ie}\email{helena.smigoc@ucd.ie}
\address[P.~Oblak]{Faculty of Computer and Information Science, University of Ljubljana, Ve\v cna pot 113, SI-1000 Ljubljana, Slovenia}
\email{polona.oblak@fri.uni-lj.si}

\subjclass[2010]{15B57, 15A18, 15B10,  05C50}
 \keywords{Symmetric matrix; Join of graphs; Inverse eigenvalue problem; Minimal number of distinct eigenvalues}
\bigskip

\maketitle 

\begin{abstract}
 We show that every $0$-$1$ multiplicity matrix for a simple graph $G$ is generically realisable for $G$. In particular, every multiplicity matrix for a path is generically realisable. We use this result to provide several families of joins of graphs that are realisable by a matrix with only two distinct eigenvalues.
\end{abstract}

\section{Introduction}
Let $G$ be a simple graph with vertex set $V(G)=\{1,\dots,n\}$ and edge set~$E(G)$, and consider $S(G)$, the set of all real symmetric $n \times n$ matrices $A=\npmatrix{a(i,j)}$ such that, for $i \neq j$, $a(i,j) \neq 0$ if and only if $\{i,j\} \in E(G)$, with no restriction on the diagonal entries of $A$. The \emph{inverse eigenvalue problem for graphs (IEP-G)}
seeks to characterise all possible spectra $\sigma(A)=\{\lambda_1,\ldots,\lambda_{n}\}$ of matrices $A\in S(G)$. The IEP-G is solved for a few families of graphs (complete graphs, paths, generalized stars~\cite{MR2022294}, cycles~\cite{MR583498,MR2549049}, generalized barbell graphs~\cite{lin2020inverse}) and for graphs of order at most five \cite{MR4074182}; the problem remains open in many other cases. 

The closely related  \emph{ordered multiplicity inverse eigenvalue problem for graphs} seeks to characterise all possible ordered multiplicities of eigenvalues of matrices in $S(G)$, i.e.,~to characterise the ordered lists of nonnegative integers $(m_1,\ldots,m_r)$ for which there exists a matrix $A \in S(G)$ and $\lambda_1<\ldots<\lambda_r$, such that $\sigma(A)=\{\lambda_1^{(m_1)},\ldots,\lambda_r^{(m_r)}\}$, where  $\lambda_i^{(m_i)}$ denotes $m_i$ copies of $\lambda_i$. We call such an ordered list $(m_1,\ldots,m_r)$ an \emph{ordered multiplicity vector of $G$}. The ordered multiplicity inverse eigenvalue problem has been resolved for all connected graphs of order at most six \cite{2017arxiv170802438}. For some specific families of connected graphs, several ordered multiplicity vectors have been determined (see e.g.~\cite{MR4044603,MR4074182,lin2020inverse}).
Moreover, Monfared and Shader proved the following theorem in~\cite{MR3506498}, showing that $(1,1,\ldots,1)$ is an ordered multiplicity vector of any connected graph $G$, which can be realised with a nowhere-zero eigenbasis, that is, a basis of eigenvectors, each containing no zero entry.

\begin{theorem}
\label{thm:nowhere0eigenbasis}
{\rm \cite[Theorem~4.3]{MR3506498}}
For any connected graph $G$ on $n$ vertices and distinct real numbers $\lambda_1,\ldots,\lambda_n$, there exists $A \in S(G)$ with spectrum $\sigma(A)=\{\lambda_1,\ldots,\lambda_n\}$ and an eigenbasis consisting of nowhere-zero vectors.
\end{theorem}

The present work draws motivation from the above result, in the context of multiplicity matrices of disconnected graphs. Such multiplicity matrices were introduced in~\cite{levene2020orthogonal} as a generalisation of ordered multiplicity vectors of connected graphs, and are defined as follows. %
Let $G$ be a graph with $k$ connected components $G_1,\dots,G_k$, and $r,k \in \bN$. An $r\times k$ matrix $V$ with non-negative integer entries
is said to be a \emph{multiplicity matrix} for $G$ if for $1\le i\le k$, the $i$th column of $V$ %
is an ordered multiplicity list realised by a matrix in $S(G_i)$. Note that a trivial necessary condition for $V$ to be a multiplicity matrix of $G$ is that the orders of the connected components of $G$ are the same as the column sums of $V$; we abbreviate this by saying that $V$ \emph{fits} $G$. 

 A matrix is said to be \emph{nowhere-zero} if each of its entries is nonzero. 
Let $\sigma=\{\lambda_1,\ldots,\lambda_n\}$ be a multiset of real numbers, and let $D \in \bR^{n \times n}$ be a diagonal matrix with diagonal elements $\lambda_i$ (in some order). For a connected graph~$G$ we say that a multiset $\sigma$ is \emph{realisable for $G$} if $\sigma=\sigma(A)$ for some $A\in S(G)$; we say $\sigma$ is \emph{generically realisable for $G$} if, moreover, for any finite set $\mc Y\subseteq \bR^n\setminus\{\0\}$ there is an orthogonal matrix $U$ so that $UD U\trans\in S(G)$, where $D$ is a diagonal matrix with spectrum $\sigma$ and $U\yy$ is nowhere-zero for all $\yy \in \mc Y$. (Note that this condition is stronger than the realisability of $\sigma$ in $S(G)$ with a nowhere-zero eigenbasis: consider $\mc Y=\{\e_1,\dots,\e_n\}$.)
An ordered multiplicity vector $\m =(m_1,\ldots,m_r)\in \bN^r$  is \emph{spectrally arbitrary} for~$G$ if for any real numbers $\lambda_1<\dots<\lambda_r$, the multiset $\sigma=\{\lambda_1^{(m_1)},\dots,\lambda_r^{(m_r)}\}$ is realisable for $G$. Further, we say $\m$ is  \emph{generically realisable for $G$} if it is spectrally arbitrary for $G$, and every assignment of eigenvalues results in a generically realisable multiset for $G$, as defined above. If every  column of  a multiplicity matrix  $V$ for a graph $G$ is generically realisable for the  corresponding connected component of $G$, then we say that $V$ is \emph{generically realisable} for $G$.

The methods used in~\cite{MR3665573,MR4074182} to resolve the IEP-G and the ordered multiplicity IEP-G for small graphs make essential use of strong properties. In particular, a symmetric $n \times n$ matrix $A$ is said to have the \emph{strong spectral property} (the \emph{SSP}) if the zero matrix $X=0$ is the only symmetric matrix satisfying $AX=XA$ and $A\circ X=I_n \circ X=0$, where $\circ$ denotes the Hadamard product. The SSP was first defined in \cite{MR3665573}. One of its key features is the following perturbation result.

\begin{theorem}
\label{thm:ssp}
{\rm \cite[Theorem~10]{MR3665573}}
Let $G'$ be a spanning subgraph of a graph $G$. If $A'\in S(G')$ is a matrix with the SSP, then for any $\epsilon>0$ there is a matrix $A\in S(G)$ with the SSP such that $A$ and $A'$ have the same eigenvalues and $\|A-A'\|<\epsilon$.
\end{theorem}

We call a matrix a \emph{$0$-$1$ matrix} if each of its entries is either $0$ or $1$. Our main result is Theorem~\ref{01GR} below, in which we apply the SSP to show that every $0$-$1$ multplicity matrix which fits a graph $G$
is generically realisable for $G$. In particular, this implies that every multiplicity matrix which fits a path $P_n$ is generically realisable for $P_n$. In the terminology of~\cite{levene2020orthogonal}, we say $P_n$ is \emph{generically realisable}.

We provide some applications in Section 3. In particular, we are interested in finding examples of joins of graphs that allow a small number of distinct eigenvalues. %
The minimum number of distinct eigenvalues of a graph
$$q(G) = \min\{q(A)\colon A \in S(G)\},$$
where $q(A)$ denotes the number of distinct eigenvalues of a square matrix~$A$, is one of the parameters motivated by IEP-G. The study of $q(G)$
was initiated by Leal-Duarte and Johnson in \cite{MR1899084}, and it has been broadly studied since then (see e.g.,~\cite{corrigendum,MR2735867,MR3665573,MR3904092,MR3118943,MR3891770,MR3506498,levene2020orthogonal}).
Monfared and Shader \cite[Theorem~5.2]{MR3506498} proved that $q(G \vee H)=2$ if $G$ and $H$ are connected graphs with $|G|=|H|$.
A consequence of our result is the following generalisation (Theorem~\ref{thm:diff=2} below): if $G$ and $H$ are arbitrary graphs, each with $k$ connected components $G_1,\dots,G_k$ and $H_1,\dots,H_k$, so that $\bigabs{|G_i|-|H_i|}\le 2$ for each $i$, then we still have $q(G\vee H)\le 2$.

\bigskip
In this paper we use the following notation. For an integer $n$, let us denote $[n]=\{1,2,\ldots,n\}$, and we also write $k+[n]:=\{k+1,k+2,\dots,k+n\}$. %
Column vectors are typically written using boldface; for example, $\1_n$ denotes the column vector of ones in $\bR^{n}$, and $\e_i:=\npmatrix{0,\dots,0,1,0,\dots,0}\trans \in \bR^n$ is the vector with the $1$ in the $i$th entry and zeros elsewhere. The $k\times k$ identity matrix is  $I_k$ is the $k \times k$ identity matrix, $0_{m \times n}$ is the zero $m \times n$ matrix, and we also write $0_m:=0_{m\times m}$ and $\0_m:=0_{m\times 1}$. Where the context allows, we may omit these subscripts altogether.

 For a vector $\xx\in \bR^n$ let us denote by $\xx(i)\in \bR^{n-1}$ the vector $\xx$ with its $i$th component removed, and for a matrix $A$ let $A(i)$ denote its principal submatrix with the $i$th row and column of $A$ removed.

 All graphs $G=(V(G),E(G))$ considered are simple undirected graphs with non-empty vertex sets $V(G)$. The \emph{order} of~$G$ is 
 $|G|=|V(G)|$ and we often assume (without loss of generality) that $V(G)= [\;|V(G)|\;]$.
 For a connected graph $G$, the \emph{distance} between a pair of vertices is the number of edges of the shortest path between them, and the \emph{diameter} $\diam(G)$ is the largest distance between any pair of vertices. A subgraph~$H$ is a \emph{spanning subgraph} of $G$ if
$V(H)=V(G)$.  The \emph{join} $G \vee H$ of two graphs $G$ and $H$ is the disjoint  union $G \cup H$ together with all the possible edges joining the vertices in $G$ to the vertices in $H$. We abbreviate the disjoint graph union of $k$ copies of the same graph $G$ by $kG:=G\cup\dots\cup G$. %
For  a subgraph $H$ of $G$ and a matrix $A\in S(G)$, we define a matrix $A[H]$ as the principal submatrix of $A$ whose rows and columns are the vertices of a subgraph $H$. 
We write $P_n$, $C_n$ and $K_n$ for the \emph{path}, the \emph{cycle} and the \emph{complete graph} on $n$ vertices, respectively,
and we denote the \emph{complete bipartite graph} on two disjoint sets of cardinalities $m$ and $n$ by $K_{m,n}:=mK_1 \vee nK_1$.

\section{Generic realisability of $0$-$1$ matrices}

In this section we will prove that any $0$-$1$ multiplicity matrix that fits a graph $G$ is generically realisable for $G$.

Throughout this section, we fix $m\in \bN$, real numbers $\lambda_1 < \lambda_2<\dots <\lambda_m$, and the diagonal matrix $\Lambda:=\diag(\lambda_1,\ldots,\lambda_m)$. Further, 
let $\El$ denote the smooth manifold of all $m \times m$ symmetric matrices with eigenvalues $\lambda_1, \dots,\lambda_m$.
This manifold was studied in \cite{MR3665573}, where a special case of the following lemma was proven (for the case the $f(\{i,j\})=1$ for all $i,j$). In fact, a nearly identical argument yields the more general result we require. We give the details for completeness.

\begin{lemma}\label{lem:SSPness}
  Let $G$ be a connected graph of order~$m$.
  Given a function $f: E(G)\rightarrow \mathbb{N}$ and $t\in \bR$, consider the family of manifolds given by
\begin{equation}\label{eq:MfG}
 \mc M_{f,G}(t):=\{A=\npmatrix{a(i,j)} \in S(G): a(i,j)=t^{f(\{i,j\})} \text{ for } \{i,j\} \in E(G)\}.
\end{equation}
For every $\epsilon >0$, there exists $t_0>0$ and a matrix $A \in \mc M_{f,G}(t_0)\cap\El$ with the SSP %
so that $\norm{A-\Lambda}<\epsilon$. 
\end{lemma}

\begin{proof}
  In this proof we use definitions and notation from  \cite{MR3665573}. In particular, $\mc N_{\mc M.X}$ denotes the normal space to a smooth submanifold $\mc M$ of the $m\times m$ matrices at some $X\in \mc M$, and two such manifolds $\mc M$ and $\mc M'$ intersect transversally at $X\in \mc M\cap \mc M'$ if $\mc N_{\mc M.X}\cap \mc N_{\mc M'.X}=\{0\}$.

  Let $\mc M(t):=\mc M_{f,G}(t)$ be the smooth family of manifolds of $m \times m$ symmetric matrices defined in \eqref{eq:MfG}. Note that $\mc M(0)$ is the set of diagonal matrices.  
 By \cite{MR3665573} we have $\mc N_{\El.\Lambda}=%
\{\alpha I_m: \alpha \in \R\}$ and $\mc N_{\mc M(0).\Lambda}=\{X: X \circ I_m=0 \}$. Since these two normal spaces have trivial intersection, the manifolds $\mc M(0)$ and $\El$ intersect transversally at $\Lambda$.

By  \cite[Theorem~3]{MR3665573} there exists $r>0$ and a continuous function $F:(-r,r)\to \El$ such  that $F(0)=\Lambda$ and for $t\in (-r,r)$, the manifolds $\El$ and $\mc M(t)$ intersect transversally at $F(t)$. Hence,  for any $\epsilon>0$, for sufficiently small $t_0> 0$, the matrix $A:=F(t_0)$ has $\|A-\Lambda\|<\epsilon$ and $A\in\mc M(t_0)\cap\El$.
To see that $A$ has the SSP, we have to prove that $S(G)$ and $\El$ intersect transversally at~$A$ (see \cite[page~11]{MR3665573}).
 Since $\mc M(t_0)\subseteq S(G)$, it follows that $\mc N_{S(G).A}\subseteq \mc N_{\mc M(t_0).A}$. Further, since $\El$ and $\mc M(t_0)$ intersect transversally, we have $\mc N_{S(G).A}\cap \mc N_{\El.A}\subseteq \mc N_{\mc M(t_0).A}\cap \mc N_{\El.A}=\{0\}$, proving that $S(G)$ and $\El$ intersect transversally at~$A$.
 \end{proof}

\begin{lemma}\label{lem:distinct}
  For $n\in \bN$, let $A_n \in \El$ and let $U_n$ be an orthogonal matrix with non-negative diagonal entries satisfying $U_n\trans A_nU_n=\Lambda$.
  If $A_n\to \Lambda$ as $n\to \infty$,  then $U_n\to I_n$ as $n\to \infty$.
\end{lemma}
\begin{proof}
  For each $k\in [m]$, let $p_k$ be a polynomial with $p_k(\lambda_\ell)=\delta_{k,\ell}$ for $\ell\in [m]$. Then $p_k(A_n)$ is the orthogonal projection onto the $\lambda_k$-eigenspace of $A_n$. Hence \[p_k(A_n)=p_k(U_n\Lambda U_n\trans)=U_np_k(\Lambda)U_n\trans=(U_ne_k)(U_ne_k)\trans\to p_k(\Lambda)=e_ke_k\trans.\] In particular, $(U_ne_k)_k(U_ne_k)_{\ell}=p_k(A_n)_{k,\ell}\to p_k(\Lambda)_{k,\ell}=\delta_{k,\ell}$. Since $(U_ne_k)_k\ge 0$, this implies that $(U_ne_k)_k\to 1$, which implies in turn that $(U_ne_k)_\ell\to 0$ if $\ell\ne k$. Hence, $U_ne_k\to e_k$, proving $U_n\to I_n$.  
\end{proof}

In the previous lemma, we showed that the off-diagonal elements of $U_n$ decay to zero. We now show that when we also have $A_n\in \mc M_{f,G}(t_n)$ where $t_n\to 0$, it is sometimes possible to precisely determine the rate of decay of off-diagonal elements of $U_n$.

\begin{lemma}\label{lem:functions}
  Let $G$ be a tree of order $m$. Given
$g:E(G)\rightarrow \mathbb{N}$, let $N_0\in \bN$ with $N_0>\max\{g(e): e \in E(G)\}\diam(G)$ and let $f:E(G)\to \bN$, $f(e):=N_0+g(e)$.
  
  For $i,j\in V(G)$, let $P(i,j)$ be the subgraph of $G$ consisting of the shortest path from $i$ to $j$, $c(i,j):=\prod_{k\in V(P(i,j))\setminus\{j\}} (\lambda_j-\lambda_k)^{-1}$, and
 $s(i,j):=\sum_{e \in E(P(i,j))}f(e)$. (In particular, $c(i,j)\ne 0$, $c(j,j):=1$ and  $s(j,j):=0$.)

Suppose $t_n>0$ with $t_n\to 0$, and $A_n \in  \mc M_{f,G}(t_n) \cap \El$ with $A_n\to \Lambda$ as $n\to\infty$. For $n\in \bN$, let $U_n=(u_n(i,j))$ be an  orthogonal matrix with non-negative diagonal entries so that $U_n\trans A_nU_n=\Lambda$. Then $ \frac{u_n(i,j)}{t^{s(i,j)}}\to c(i,j)$ as $n\to \infty$, for all $i,j\in V(G)$.%
\end{lemma}

\begin{proof}
  Assuming $V(G)=[m]$ and $A_n=\npmatrix{a_n(i,j)} \in \mc M_{f,G}(t_n)$, it follows that $a_n(i,j)=t^{f(\{i,j\})}$ for each $\{i,j\}\in E(G)$. Since $0<t_n\to 0$ and $A_n\to \Lambda$, we may assume (by taking $n$ sufficiently large) that $0<t_n<1$ and for $i\ne j$, $a_n(i,i)- \lambda_j$ is bounded away from zero.

  Fix $j_0\in \bN$ and consider the vector $U_ne_{j_0}=\npmatrix{u_n(i,j_0)}_{i\in [m]}$. This is a normalised eigenvector of $A_n$ with eigenvalue $\lambda_{j_0}$. Equivalently, for $i\in [m]$ we have 
\begin{equation}\label{eq:term}
(\lambda_{j_0}-a_n(i,i)) u_n(i,j_0)= \sum_{k \in N_G(i)} t_n^{f(\{i,k\})}u_n(k,j_0)%
\end{equation}
where $N_G(i)$ is the set  of neighbours of $i$ in $G$.

For $i\in V(G)$, let $d(i,j_0):=| E(P(i,j_0))|\ge0$ denote the distance in $G$ from $i$ to $j_0$. We claim that for $0\le x\le \diam(G)$:
\begin{enumerate}
\item\label{a} if $i\in V(G)$ with $d(i,j_0)>x$, then $\frac{u_n(i,j_0)}{t_n^{x N_0}}\to 0$ as $n\to \infty$; and
\item\label{b} if $i\in V(G)$ with $d(i,j_0)=x$, then $\frac{u_n(i,j_0)}{t_n^{s(i,j_0)}}\to c(i,j_0)$ as $n\to \infty$.\end{enumerate}
We will establish this claim by induction on $x$.

Since $A_n$ converges to $\Lambda$, the $j_0$-th column of $U_n$ converges to $e_{j_0}$ by Lemma~\ref{lem:distinct}. %
This implies that the claim holds for $x=0$.

Now assume inductively that the claim holds for all $x$ with $0 \leq x\leq x_0<\diam(G)$. We proceed to prove that it holds for $x=x_0+1$ as well.
First consider claim~\ref{a}, and suppose %
$i\in V(G)$ with $d(i,j_0)>x_0+1$. %
Rearranging \eqref{eq:term}, we obtain:
$$ \frac{u_n(i,j_0)}{t_n^{(x_0+1)N_0}}=\frac1{\lambda_{j_0}-a_n(i,i)} \sum_{k \in N_G(i)}\frac{u_n(k,j_0)}{t_n^{(x_0+1)N_0-f(\{i,k\})}}.$$
For any $e\in E(G)$, we have $f(e)> N_0$, so $(x_0+1)N_0-f(e)<x_0N_0$, and: 
\begin{align*} \left|\frac{u_n(i,j_0)}{t_n^{(x_0+1)N_0}}\right|&\le \frac1{|\lambda_{j_0}-a_n(i,i)|} \sum_{k \in N_G(i)}\left|\frac{u_n(k,j_0)}{t_n^{(x_0+1)N_0-f(\{i,k\})}}\right|\\
                                                               &\le \frac1{|\lambda_{j_0}-a_n(i,i)|} \sum_{k \in N_G(i)}\left|\frac{u_n(k,j_0)}{t_n^{x_0N_0}}\right|.
\end{align*}
Since $d(i,j_0)>x_0+1$, we have $d(k,j_0)>x_0$ for all $k\in N_G(i)$, hence by part~(a) of the claim for $x=x_0$, each term in the sum above converges to $0$. Since we also have $\lambda_{j_0}-a_n(i,i)\to \lambda_{j_0}-\lambda_{i}\ne 0$,
part \ref{a} of the claim holds for $x=x_0+1$. This proves  \ref{a} for all $x$ with $0\le x\le \diam(G)$.

Now consider part \ref{b}. Suppose $d(i,j_0)=x_0+1$. Since $G$ is a tree, there exists precisely one $k_0 \in N_G(i)$ with $d(k_0,j_0)=x_0$, and by the definition of $s$, we have $s(i,j_0)=f(\{i,k_0\})+s(k_0,j_0)$. We have $i\ne j_0$, so $\lambda_{j_0}-a_n(i,i)\ne 0$ and we can rearrange~\eqref{eq:term}as follows:
\begin{equation}\label{eq:divide}
 \frac{u_n(i,j_0)}{t_n^{s(i,j_0)}}=\frac1{\lambda_{j_0}-a_n(i,i)}\left( \frac{u_n(k_0,j_0)}{t_n^{s(k_0,j_0)}}
+\sum_{k \in N_G(i)\setminus \{k_0\} }\frac{u_n(k,j_0)}{t_n^{s(i,j_0)-f(\{i,k\})}}\right).
\end{equation}
For every $k \in N_G(i)\setminus \{k_0\}$ we have:
\begin{align*}
 s(i,j_0)-f(\{i,k\})&=\left(\sum_{e \in E(P(i,j_0))}f(e)\right)-f(\{i,k\})\\
 &=(x_0+1)N_0+\left(\sum_{e \in E(P(i,j_0))}g(e)\right)-N_0-g(\{i,k\}) \\
 &< x_0N_0+(x_0+1)\max\{g(e): e \in E(G)\} \\ &< (x_0+1) N_0.
\end{align*}
Hence, by part~\ref{a} of the claim, all the terms under the sum in \eqref{eq:divide} converge to zero. Taking limits and using part~\ref{b} of the claim shows that $\frac{u_n(i,j_0)}{t_n^{s(i,j_0)}}\to (\lambda_{j_0}-\lambda_{i})^{-1}c(k_0,j_0)=c(i,j_0)$. 
\end{proof}

\begin{theorem}\label{thm:genericU}
  Let $G$ be a connected graph of order~$m$,
  and choose a finite set $\mc Y \subset \R^m \setminus \{\0\}$. There exists a matrix $A \in S(G)\cap \El$ with the SSP and an orthogonal matrix $U$ with $U\trans AU=\Lambda$ so that $U\vv$ is nowhere-zero for all $\vv \in \mc Y$.  
\end{theorem}

\begin{proof}
First we prove the statement for trees, so let $G'$ be a tree. 
Let us assume that the statement does not hold. Choose $g: E(G')\to \bN$ and $f:=N_0+g$ as in Lemma~\ref{lem:functions} so that the corresponding function $s:V(G)\times V(G)\to \bN$ satisfies $s(i,j)=s(k,\ell)\iff \{i,j\}=\{k,\ell\}$. (This may be guaranteed by a suitable choice of $g$.) Further, let $A_n \in \mc M_{f,G'}(t_n) \cap \El$ also be as in Lemma \ref{lem:functions}, with the SSP (which we can guarantee by Lemma~\ref{lem:SSPness}) and let $U_n$ be orthogonal matrices with non-negative diagonal satisfying $U_n\trans A_nU_n=\Lambda$.  By assumption, for every $n$ there exists $\vv \in \mc Y$ so that some entry of $U_n \vv$ is zero. Passing to a subsequence, we may assume that there is some fixed vector $\vv=(v_1,\dots,v_m)\trans \in \mc Y$ and a fixed index $k\in [m]$ so that $(U_n \vv)_k=0$ for every $n\in \bN$. %
Let $\vu_n=\npmatrix{u_n(k,1) & u_n(k,2) & \ldots & u_n(k,m) }$ denote the $k$-th row of $U_n$, hence we are assuming $\vu_n\trans \vv=0$ for all $n\in \bN$.   We aim to arrive at the contradiction by proving  $\vv=\0$. 
From Lemma \ref{lem:functions} we know that $U_n\to I$, hence $v_k=0$. Let $i_1, i_2,\ldots, i_{m-1} \in [m]\setminus \{k\}$ be such that $s(k,i_1) < s(k, i_2)<\dots<s(k,i_{m-1})$, and $\ell\in [m-1]$ be minimal with $v_{i_\ell} \neq 0$. Now: 
$$0=\frac{1}{t_n^{s(k,i_\ell)}}\vu_n\trans \vv=\sum_{r=\ell}^{m-1} \frac{u_n(k,i_r)}{t_n^{s(k,i_\ell)}}v_{i_r}.$$
By Lemma \ref{lem:functions}, we know that $\frac{u_n(k,i_\ell)}{t_n^{s(k,i_\ell)}}$ converges to a nonzero constant and $\frac{u_n(k,i_r)}{t_n^{s(k,i_\ell)}}$ converges to zero for $r>\ell$. Hence $v_{i_\ell}=0$, %
a contradiction.

Now, let $G$ be a connected graph, and $G'$ be a spanning tree of $G$.  Since the claim holds for $G'$, there exists $A' \in S(G')$ with the SSP so that ${U'}\trans A'U'=\Lambda$, ${U'}\trans U'=I_m$, and $U'\vv$ is nowhere-zero for all $\vv \in \mc Y$. Since $A'$ has the SSP, by Theorem \ref{thm:ssp} there exists $A \in S(G)$ with the SSP arbitrarily  close to $A'$ with the same eigenvalues as $A'$.  Since the eigenvalues are distinct, the $\lambda_k$ eigenspaces are close for $A$ and $A'$, so for $A$ and $A'$ sufficiently close, there is an orthogonal matrix $U$ diagonalising $A$ which is sufficiently close to $U'$ to ensure that $U\vv$ is also nowhere zero for all $\vv\in \mc Y$.\end{proof}

\begin{theorem}\label{01GR}
If $G$ is any graph, then every $0$-$1$ multiplicity matrix which fits $G$ is generically realisable for $G$.
\end{theorem}

\begin{proof}
  Given such a matrix $V$, label the connected components $G_1,\dots,G_k$ of $G$ so that the multiplicity vector $V\e_i$ fits $G_i$. Every entry of $V\e_i$ is $0$ or $1$, so by Theorem~\ref{thm:genericU}, $V\e_i$ is generically realisable for $G_i$. Hence, $V$ is generically realisable for $G$.
\end{proof}

If every multiplicity matrix for a graph~$G$ is generically realisable for~$G$, then we say that $G$ is \emph{generically realisable}. 
This is a strong requirement for a graph; in particular, it implies that $G$ is spectrally arbitrary for every multiplicity matrix that can be realised by $G$. In fact, the only families of generically realisable graphs previously known are unions of complete graphs~\cite{levene2020orthogonal}. We can now extend this to include paths.

\begin{corollary}\label{thm:pathsGR}
  \begin{enumerate}
  \item The path $P_n$ is a generically
    realisable graph, for every $n\in \bN$.
  \item If every connected component of a graph $G$ is either a complete graph or a path, then $G$ is generically realisable.
\end{enumerate}
\end{corollary}
\begin{proof}
  The maximum multiplicity of an eigenvalue of a path is one~\cite{MR0244285}, so every
  multiplicity vector for~$P_n$ is a $0$-$1$ vector, hence it is generically realisable by Theorem~\ref{01GR}. Complete graphs are also generically realisable~\cite{levene2020orthogonal}, so the second assertion follows immediately.
\end{proof}

\section{Applications}

In \cite{levene2020orthogonal} the authors developed an approach to study $q(G)$, where $G$ is the join of two graphs. We now give some applications of preceding results to this problem.

First we briefly review the set up from \cite{levene2020orthogonal}.
For a matrix $X$ with at least $3$ rows, we write $\widetilde X$ for the matrix obtained by deleting the first row and the last row of $X$.  Let $\mathcal{S}^{m\times n}$ denote the set of $m\times n$ matrices with entries in a set $\mathcal S$, and let $\bN_0$ be the set of non-negative integers. Given $r,k,\ell\in \bN$ with $r\ge3$, two matrices $V \in \bN_0^{r \times k}$ and $W \in \bN_0^{r \times \ell}$ are said to be \emph{compatible}
  if $\widetilde V\1_k=\widetilde W\1_\ell$ and $\widetilde V\trans \widetilde W$ is nowhere-zero. We say that two graphs $G,H$ \emph{have compatible multiplicity matrices} if there exist compatible matrices $V,W$ where $V$ is a multiplicity matrix for $G$ and $W$ is a multiplicity matrix for $H$.
  
  Which graphs $G$ and $H$ have $q(G\vee H)=2$? The answer is closely related to the existence of compatible multiplicity matrices.

\begin{theorem}\cite[Theorems~2.5 and~2.14]{levene2020orthogonal}\label{thm:q=2}
Let $G$ and $H$ be two graphs. If $q(G\vee H)=2$, then $G$ and $H$ necessarily have compatible multiplicity matrices. 

Moreover, if there exist compatible generically realisable multiplicity matrices $V$ and $W$ for $G$ and $H$, then $q(G\vee H)= 2$.
\end{theorem}

By Corollary~\ref{thm:pathsGR},  in the case of unions of paths or complete graphs we can strengthen this to a necessary and sufficient condition.

\begin{corollary}\label{cor:PK-compat-iff}
  If $G$ and $H$ are unions of paths or complete graphs, then $q(G\vee H)= 2$ if and only if $G$ and $H$ have a pair of compatible multiplicity matrices.
\end{corollary}

For general graphs, by combining Theorem~\ref{01GR} and Theorem~\ref{thm:q=2}, we obtain the following purely combinatorial sufficient condition for $q(G\vee H)$ to be $2$.

\begin{corollary}\label{cor:P-to-connected-q=2}
Let $G$ and $H$ be two graphs.
 If there exist compatible $0$-$1$ matrices $V$ and $W$ that fit $G$ and $H$, respectively, then $q(G\vee H)=2$.
\end{corollary}

In general, deciding whether there exist compatible $0$-$1$ matrices with prescribed column sums seems to be a difficult combinatorial question, which we plan to examine further in upcoming work.  %

Monfared and Shader \cite[Theorem~5.2]{MR3506498} proved that $q(G \vee H)=2$ if $G$ and $H$ are connected graphs with $|G|=|H|$.
Using Corollary~\ref{cor:P-to-connected-q=2}, it is now simple to generalise their result to the case $\bigabs{|G|-|H|}\leq 2$, and also to pairs of disconnected graphs with equal numbers of connected components. %

\begin{theorem}\label{thm:diff=2}
  Let $k\in \bN$, and for $i\in [k]$, let $G_i$ and $H_i$ be connected graphs with $\bigabs{|G_i| -|H_i|}\leq 2$. Then $$q(\bigcup_{i=1}^k G_{i} \vee  \bigcup_{i=1}^k H_{i})=2.$$ %

Moreover, if  $m_{i}+2 \geq |H_i|$ for all  $i \in [k]$, then
 $$q(\bigcup_{i=1}^k K_{m_i} \vee  \bigcup_{i=1}^k H_{i})=2.$$ 
 \end{theorem}

\begin{proof}
  Let $p_i:=\min\{|G_i|,|H_i|\}$ for $i \in [k]$, $p:=\max \{p_i\colon i\in [k]\}$ and ${\bf q}_i:=\sum_{j \in [p_i]} \e_i \in \{0,1\}^{p}$.
  Let $E:=\npmatrix{
 {\bf q}_1 & {\bf q}_2 & \cdots & {\bf q}_k } 
 \in \{0,1\}^{p \times k}$
 and consider
  \begin{equation}\label{eq:VW k components}
  V:=\npmatrix{{\bf v}_1 \\ E \\ {\bf v}_{p+2}}\in \{0,1\}^{(p+2) \times k} \text{ and } W:=\npmatrix{{\bf w}_1 \\ E \\ {\bf w}_{p+2}}\in \{0,1\}^{(p+2) \times k},
  \end{equation}
  where ${\bf v}_1,{\bf v}_{p+2},{\bf w}_1,{\bf w}_{p+2}\in \{0,1\}^{1\times k}$ are chosen so that $\1_{p+2}\trans V \e_i=|G_i|$, and $\1_{p+2}\trans W \e_i=|H_i|$ for all $i\in [k]$. (The existence of such vectors is assured by the condition $\bigabs{|G_i|-|H_i|}\le2$.) Since $\widetilde V=\widetilde W=E$ and the first row of $E$ is nowhere-zero, $V$ and $W$ are compatible $0$-$1$ matrices fitting $\bigcup_{i=1}^k G_{i}$ and $\bigcup_{i=1}^k H_{i}$, respectively, hence  $q(\bigcup_{i=1}^k G_{i} \vee  \bigcup_{i=1}^k H_{i})=2$ by Corollary~\ref{cor:P-to-connected-q=2}.

  If $G=\bigcup_{i=1}^kK_{m_i}$, then the assumption $m_i \geq |H_i|-2$ for all $i \in [k]$ implies the existence of compatible %
  matrices $V$ and $W$ as in \eqref{eq:VW k components}, except that we don't require the first and the last row of $V$ to be $0$-$1$ vectors, i.e.~${\bf v}_1,{\bf v}_{p+2} \in \bN_0^k$. By \cite{levene2020orthogonal} and Theorem~\ref{01GR}, $V$ and $W$ are generically realisable multiplicity matrices for $G$ and $H$, respectively, so $q=2$ by Theorem~\ref{thm:q=2}.
 \end{proof}

For connected graphs $G$ and $H$  we have just seen that $q(G\vee H)=2$ if $\bigabs{|G|-|H|}\le 2$.  We now show that we generally cannot relax this inequality.

\begin{example}\label{ex:PjoinH}
 Suppose that $G=P_n$ and $H$ is a connected graph with $|H|\le n$. 
 Let us prove that in this case the condition %
 $\bigabs{|G|-|H|}\le 2$, i.e., $|H|\ge n-2$,
 is also necessary for $q(P_n \vee H)=2$.
 
 If $q(P_n\vee H)=2$, then by Theorem~\ref{thm:q=2} there exist compatible multiplicity matrices $V\in \bN_0^{(r+2) \times 1}$ and $W \in \bN_0^{(r+2) \times 1}$ for $P_n$ and  $H$, respectively.  The maximum multiplicity of a path is~$1$, so $V\in \{0,1\}^{(r+2)\times 1}$. By compatibility, %
 we have $\widetilde V=\widetilde W$
 and 
 without loss of generality, we can delete any zeros in these column vectors to reduce to the case 
 $$ V=\npmatrix{v_1 \\ \1_{r}\\ v_{r+2}} \text{ and } W=\npmatrix{w_{1} \\ \1_{r} \\ w_{r+2}},$$
 where  $v_1,v_{r+2} \in  \{0,1\}$ and $w_{1},w_{r+2}\in \bN_0$. Since these matrices fit $G$ and $H$, we have $r=n-(v_1+v_{r+2})\geq n-2$ and $|H|\geq r\geq n-2$, as required.
 
 Recall that if $T$ is a tree and $A\in S(T)$, then the extreme eigenvalues of $A$ have multiplicity one. Hence, by a similar argument to that of the previous paragraph, we have $q(T_1\vee T_2)=2\iff \bigabs{|T_1|-|T_2|}\le 2$ for any trees $T_1,T_2$.

\end{example}

We now turn to another class of applications of Theorem~\ref{01GR}, in which we obtain certain achievable spectra of partial joins.  %
We require additional terminology given below.

Recall that if $G$ is a graph and $W\subseteq V(G)$, then the \emph{vertex boundary} %
of $W$ in $G$ is the set of all vertices in $V(G)\setminus W$ which are joined to some vertex of $W$ by an edge of $G$.

Suppose that $G_1$ and $G_2$ are two disjoint  graphs and let $V_i \subseteq V(G_i)$ for $i=1,2$. The \emph{partial join} of graphs $G_1$ and $G_2$ is the graph $(G,V)=(G_1,V_1)\vee (G_2,V_2)$ formed from $G_1 \cup G_2$ by joining each vertex of $V_1$ to  each vertex of $V_2$.  If $V_2=V(G_2)$, then we write $(G_1,V_1)\vee G_2:= (G_1,V_1)\vee (G_2,V(G_2))$. 

Suppose $G$ and $G'$ are graphs so that $G$ is obtained from $G'$ by deleting $s$ vertices and adding arbitrary number of edges (in particular, $|G'|=|G|+s$). If a matrix $A\in S(G)$ has the SSP, then by  the Minor Monotonicity Theorem \cite[Theorem~6.13]{MR4074182} there exist a matrix $A' \in S(G')$ with the SSP, such that $\sigma(A')=\sigma(A) \cup \{\mu_1,\ldots,\mu_s\}$, where $\mu_i\ne \mu_j$ for distinct $i,j \in [s]$. In the next result we provide a statement of a similar flavour, that does not depend on the SPP. The construction given the Lemma was first developed in the context of the nonnegative inverse eigenvalue problem \cite{MR2133312,MR2372590}.

\begin{lemma}\label{lem:HisGR}
 Let $G$ be a graph, $V(G)=V_1\cup V_2$ a partition of $V(G)$, and $X\subseteq V_1$ the vertex boundary of $V_2$ in~$G$. Suppose $H$ is any connected graph with $|H|\ge |V_2|$ and consider the graph $G':=(G[V_1],X)\vee H$. %

 Let $M \in S(G)$  and suppose that $\sigma'$ is a multiset of real numbers so that %
 $\sigma(M[V_2]) \cup \sigma'$ is generically realisable for $H$. Then there exists $N \in S(G')$ with spectrum  $\sigma(N)=\sigma(M) \cup \sigma'$.
  \end{lemma}

\begin{proof} Let us denote $m_i:=|V_i|$ for $i=1,2$ and $k:=|X|$. Assume without loss of generality that $V(G)=[m_1+m_2]$, $V_1=[m_1]$ and $X=m_1-k+[k]$.
We have $M=\left(
\begin{array}{cc}
 A & B \\
 B\trans & C \\
\end{array}
\right) \in S(G)$, with $A=M[V_1]$, $C=M[V_2]$, and $\sigma(C)=\{\lambda_1,\ldots,\lambda_{m_2}\}$.  Let $Q\in \bR^{m_2\times m_2}$ be an orthogonal matrix, such that $Q\trans CQ=\Lambda_0:=\diag(\lambda_1,\ldots,\lambda_{m_2})$.  Write $\sigma'=\{\mu_1,\ldots,\mu_t\}$, and let $\Lambda':=\diag(\mu_1,\ldots,\mu_t)$. Moreover, let $N'=M \oplus \Lambda'$ with eigenvalues $\sigma(M) \cup \sigma'$.

Since $X=m_1-k+[k]$ is the vertex boundary of $V_2$ in $G$, we have $B\trans=\npmatrix{ 0_{m_2 \times (m_1-k)}&B_0\trans } \in \bR^{m_2 \times m_1}$, where no column of $B_0\trans \in \bR^{m_2 \times k}$ is zero. Let us denote the set of columns of $Q\trans B_0\trans$ by $\mathcal X \subset \R^{m_2}$, and let $\mathcal Y \subset \R^{m_2+t}$ denote the set of vectors obtained from elements of $\mathcal X$ by appending $t$ zeros. Then $|{\mathcal X}|=|{\mathcal Y}|=k$. Since $Q$ is orthogonal, all vectors in $\mathcal Y$ are nonzero.

Since $\sigma(M[V_2]) \cup \sigma'$ is generically realisable for $H$, there exist a matrix $C' \in S(H)$ with $\sigma(C')=\sigma(C)\cup \sigma'$, and an orthogonal matrix $U\in \bR^{(m_2+t)\times (m_2+t)}$, such that $U^TC'U=\Lambda_0 \oplus \Lambda'$ and  $U\yy$ is a nowhere-zero vector for all $\yy \in \mc Y$.

 Let 
 $$U':=I_{m_1} \oplus \left( U (Q\trans \oplus I_{t})\right) \text{ and }
N:=U' N'  {U'}\trans=\npmatrix{
A & Z \\
 Z\trans & C' },$$
where \[Z\trans=\npmatrix{0_{(m_2+t)\times (m_1-k)}&U\npmatrix{ Q\trans B_0\trans\\ 0_{t \times k}}}.\] Since $U\yy$ is a nowhere-zero vector for $\yy \in {\mathcal Y}$, the matrix $U\npmatrix{ Q\trans B_0\trans\\ 0_{t \times k}}$ is nowhere-zero. Hence, $N \in S(G')$, and $\sigma(N)=\sigma(N')=\sigma(M)\cup \sigma'$ as desired. 
\end{proof}

Lemma \ref{lem:HisGR} can be used to build explicit examples or to provide more general results. Implementation is faced with two issues: first we need some information on $\sigma(M[V_2])$, and second, we need to prove generic realisability of $\sigma(M[V_2]) \cup \sigma'$ for $H$. In our first application, we rely on Theorem \ref{01GR} for generic realisably. 

\begin{theorem}\label{thm:partial join}
  Let $G$ be a graph, $V(G)=V_1\cup V_2$ a partition of $V(G)$, and $X\subseteq V_1$ be the vertex boundary of $V_2$ in~$G$. Suppose $H$ is any connected graph with $|H|\ge |V_2|$ and consider the graph $G':=(G[V_1],X)\vee H$. %

  If there exists $M \in S(G)$ so that  $M[V_2]$ has distinct eigenvalues $\lambda_1,\ldots,\lambda_{|V_2|}$, then there exists $N \in S(G')$ with spectrum 
  $$\sigma(N)=\sigma(M) \cup \{\mu_1,\ldots,\mu_t\},$$ where  $t=|H|-|V_2|$
 and $\mu_1,\ldots,\mu_t$ are any distinct real numbers satisfying $\mu_i \ne \lambda_j$ for all $i \in [t], j\in [|V_2|]$. 
\end{theorem}

Theorem~\ref{thm:partial join} allows us, in some sense, to replace a submatrix with distinct eigenvalues with a matrix corresponding to an arbitrary connected graph of equal or bigger size. Note, the technical requirement that $M[V_2]$ has distinct eigenvalues is automatically satisfied when $G[V_2]$ is a path. This observation allows us to improve an upper bound for the number of distinct eigenvalues of  joins of graphs that is given in \cite[Section~4.2]{MR3904092}, where it is proved that for connected graphs $G$ and $H$, $q(G \vee H)$ is bounded above by $2+\bigabs{|G|-|H|}$. 

\begin{corollary}\label{cor:join}
  For any $n\in \bN_0$ and a %
  graph $H$ with $|H|\geq n$ we have
$$q(G \vee H) \leq q(G \vee P_n)+|H|-n.$$

In particular, if $G$ and $H$ are both connected, then
$$q(G \vee H) \leq \max\left\{2,\bigabs{|G|-|H|}\right\}.
$$
\end{corollary}

\begin{proof}
 The first part follows easily from Theorem~\ref{thm:partial join}. %
 Assume now that both $G$ and $H$ are connected. If $\bigabs{|H|-  |G|}\leq 2$, then $q(G \vee H)=2$ by Theorem~\ref{thm:diff=2}. To cover the case $|H|\geq |G|+3$, we note that Theorem~\ref{thm:diff=2} implies $q(G \vee P_{|G|+2})=2$ and hence we have $q(G \vee H) \leq 2+|H|-(|G|+2)=|H|-|G|$. By symmetry the statement follows.
\end{proof}

Next we explore applications of Lemma \ref{lem:HisGR} in the special case where $G$ is taken to be a cycle. Cycles are a nice family to use for this purpose, since every induced connected subgraph of a cycle is a path, and the IEP-G for cycles is known to have the following solution.

\begin{lemma}\cite[Theorem~3.3]{MR2549049}\label{lemma:IEPGcycles}
 Let $\lambda_1\leq \lambda_2\leq \ldots \leq \lambda_n$ be a list of real numbers, $n\geq 3$. Then $\{\lambda_1, \ldots, \lambda_n\}$ is the set of eigenvalues of a matrix in $S(C_n)$ if and only if 
 $$\lambda_1\leq \lambda_2 < \lambda_3\leq \lambda_4<\ldots %
 \; \text{ or }\;
 \lambda_1< \lambda_2 \leq \lambda_3< \lambda_4\leq\ldots %
 .$$
\end{lemma}

\begin{example}
  Fix $n\in \bN$ and let $X$ be the set of two degree one vertices of $P_n$. Given a connected graph $H$, consider $J:=(P_n,X)\vee H$. We claim that if $s\in \bN_0$ and $2s\leq|J|= n+|H|$ then $\{2^{(s)},1^{(|J|-2s)}\}$ is an unordered multiplicity list for some matrix in $S(J)$. Let $G=C_{n+|H|}$. By Lemma~\ref{lemma:IEPGcycles}, there is a matrix in $S(G)$ with unordered multiplicity list $\{2^{(s)},1^{(|J|-2s)}\}$. Partition the vertices of $G$ as $V_1\cup V_2$ so that $P_n=G[V_1]$ and $P_{|H|}=G[V_2]$, and apply Theorem~\ref{thm:partial join}.

\begin{figure}[h]
\begin{center}
\begin{tikzpicture}
\draw (3,0) ellipse (1cm and 2cm);
\node[rectangle, draw=none, right, align=left] at (3,0) {$H$}; 

\node[fill=gray]  (6)  at (1,1) {}; 
\node[fill=gray]  (0)  at (1,-1) {}; 

\foreach \x in {-1.5,-1,-0.5,0,0.5,1,1.5}{
            \draw (3,\x)--(0);
            \draw (3,\x)--(6);
    }

\node (1)  at (0,-1) {}; 
\node (2)  at (-1,-1) {}; 
\node (3)  at (-2,0) {}; 
\node (4)  at (-1,1) {}; 
\node(5)  at (0,1) {}; 

       \foreach \x in {1,2,3,4,5}{
                    \node at (\x) { };
  } 
  \draw (0) -- (1) -- (2) --(3)--(4)--(5)--(6);

\end{tikzpicture}
\end{center}
\caption{Sketch of the partial join $(P_7,X) \vee H$, where pendant vertices $X$ of $P_7$ are coloured grey. The edges of $H$ are not drawn.}
\label{fig:cycle}
\end{figure}
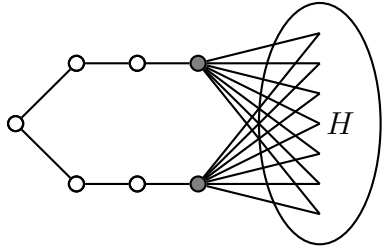
\end{example}

Note that Lemma \ref{lem:HisGR} allows us to increase the multiplicites of  eigenvalues of $M$, provided the technical conditions of the lemma are satisfied. For example, the eigenvalues $\{\mu_1, \ldots, \mu_t\}$ that we add in Lemma~\ref{lem:HisGR} can be chosen to agree with eigenvalues in $\sigma(M)$, hence increasing their multiplicity,  provided they are not also eigenvalues of $M[V_2]$. When this condition is not satisfied, Theorem \ref{01GR} cannot be applied to assure generic realisability. Examples of both eventualities are given in the two examples below. In the first example we exhibit ordered multiplicity lists that achieve $q(G)$, where $G$ is a wheel graph of order $4m-1$.

\begin{example}\label{ex:K1veeH}
 Let us apply Lemma \ref{lem:HisGR} to $G=C_{2m}$, $X=V_1=\{2m\}$ and $V_2=[2m-1]$. Choose any distinct real numbers $\lambda_1,\ldots,\lambda_{m}$. By  Lemma \ref{lemma:IEPGcycles}     there exists a matrix $M \in S(C_{2m})$ with  $$\sigma(M)=\{\lambda_1^{(2)},\lambda_2^{(2)},\ldots,\lambda_{m}^{(2)}\}.$$
 Let $B =M(2m)\in S(P_{2m-1})$ be the leading principal $(2m-1) \times (2m-1)$ submatrix of $M$. By interlacing and~\cite{MR0244285} we have $$\sigma(B)=\{\lambda_1,\mu_1,\lambda_2,\mu_2,\ldots,\mu_{m-1},\lambda_{m}\},$$ $\mu_i \in (\lambda_i,\lambda_{i+1})$ for $i\in [m-1]$.

 Let $H=C_{4m-2}$ and let $\sigma'=\sigma(B)$. By Lemma \ref{lemma:IEPGcycles} there exists $A \in S(C_{4m-2})$ with  $\sigma(A)=\sigma(B) \cup \sigma'=\{\lambda_1^{(2)},\mu_1^{(2)},\lambda_2^{(2)},\mu_2^{(2)},\ldots,\mu_{m-1}^{(2)},\lambda_{m}^{(2)}\}$. Let us prove that we can choose a nowhere-zero eigenbasis for $A$. Every eigenspace ${\mathcal V}$ of $A$ is spanned by two linearly independent vectors $\vv=\npmatrix{v_i}%
 $ and $\ww=\npmatrix{w_i}%
 $.   If  $v_i=w_i=0$ for some $i\in [4m-2]$, then $\vv(i)$ and $\ww(i)$ are linearly independent eigenvectors of $A(i)\in S(P_{4m-3})$ corresponding to the same eigenvalue, which contradicts the fact every matrix corresponding to a path has simple eigenvalues, \cite{MR0244285}.  Hence for each $i$ either $v_i \ne 0$ or $w_i\ne 0$, so we can choose their linear combinations to be nowhere-zero and hence we can choose  a nowhere-zero eigenbasis for $A$. 
 
 Since $A$ has a nowhere-zero eigenbasis and no simple eigenvalues, the proof of \cite[Proposition~3.4]{levene2020orthogonal} shows that  $\sigma(A)=\sigma(B) \cup \sigma'$ is generically realisable for $C_{4m-2}$.  By Lemma~\ref{lem:HisGR}, there exists a matrix $N \in S(K_1\vee C_{4m-2})$ with spectrum 
  $$\sigma(M)\cup \sigma'=\{\lambda_1^{(3)},\mu_1,\lambda_2^{(3)},\mu_2,\ldots,\mu_{m-1},\lambda_{m}^{(3)}\} .$$
  Hence the ordered multiplicity list $(3,1,3,1,\ldots,3)$ is realisable for $K_1 \vee C_{4m-2}$, which are also known as wheel graphs. 
  
  Note that Lemma \ref{lemma:IEPGcycles} prohibits odd number of eigenvalues between any two double eigenvalues of a matrix corresponding to a cycle. Hence by interlacing a matrix corresponding to $K_1 \vee C_{4m-2}$ cannot have an even number (including zero) of eigenvalues between any two triple eigenvalues. Therefore with the above construction we have found a matrix corresponding to $K_1 \vee C_{4m-2}$ with the maximal multiplicity of an eigenvalue and the minimal number of distinct eigenvalues. In particular, $q(K_1 \vee C_{4m-2})=2m-1$.
\end{example}

\begin{example}
  If $m\ge2$ and $H$ is any connected graph with $3m-2$ vertices, then the multiplicity list $\{3^{(m)}\}$ is spectrally arbitrary for $K_2\vee H$ and, in particular, $q(K_2\vee H)\le m$. To see this, 
  take $G=C_{2m}$, $X=V_1=\{2m,2m-1\}$ and $V_2=[2m-2]$.
For any $\lambda_1<\dots<\lambda_m$, let $M \in S(C_{2m})$ with  $\sigma(M)=\{\lambda_1^{(2)},\lambda_2^{(2)},\ldots,\lambda_{m}^{(2)}\}$. Then 
$\sigma(M(2m))=\{\lambda_1,\mu_1,\lambda_2,\mu_2,\ldots,\mu_{m-1},\lambda_{m}\}$, where $\mu_i \in (\lambda_i,\lambda_{i+1})$ for $i\in [m-1]$ (since a path has only simple eigenvalues).
Now, $M[V_2]$ is a principal submatrix of $M(2m)$, so by~\cite[Problem 4.3.P17]{MR2978290}, the eigenvalues of $M[V_2]$ strictly interlace those of $M(2m)$. In particular, the eigenvalues of $M[V_2]$ do not intersect $\sigma':=\{\lambda_1, \ldots,\lambda_m\}$. %
Applying Theorem \ref{thm:partial join} we see that the spectrum $\sigma(M)\cup \sigma'=\{\lambda_1^{(3)},\lambda_2^{(3)},\ldots,\lambda_{m}^{(3)}\}$ is realised by a matrix in $S(K_2 \vee  H)$, as required.

\end{example}

\begin{remark}
  In particular, by interlacing, the maximal multiplicity of an eigenvalue of a matrix correspoding to $P_n \vee K_2$ is equal to $3$ and hence $q(P_{3m-2} \vee K_2)=m$. This result complements \cite[Example~4.5]{MR3904092}, where they proved that $q(P_{s} \vee K_1)=\lfloor\frac{s+1}{2}\rfloor$. It would be interesting to determine $q(P_{s} \vee K_t)$ in general.
\end{remark}

In Theorem \ref{thm:partial join} we can make use of graphs for which IEP-G is solved (or better understood) to construct new realisable multiplicity lists for partial joins; we illustrate this idea also for generalised stars \cite{MR2022294}.

\begin{example}
Let $k \in\bN$  and  $H$  an arbitrary connected graph. If  ${\underline m}=\{m_1,m_2,\ldots,m_\ell,1^{(n)}\}$, such that $m_i \geq 2$, $\ell<n$, and $\sum_{i\in [\ell]} m_i=|H|-n+k+1\leq k+\ell$, then ${\underline m}$ is an unordered multiplicity list for $K_1 \vee (k K_1 \cup H)$.

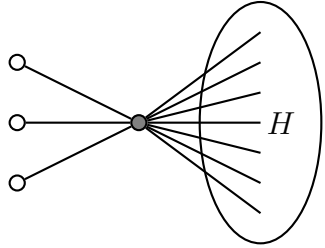
\begin{figure}[h]
\begin{center}
\begin{tikzpicture}
\draw (3,0) ellipse (1cm and 2cm);
\node[rectangle, draw=none, right, align=left] at (3,0) {$H$}; 

\node[fill=gray]  (0)  at (1,0) {}; 

\node (1)  at (-1,0) {}; 
\node (2)  at (-1,1) {}; 
\node (3)  at (-1,-1) {}; 

\foreach \x in {-1.5,-1,-0.5,0,0.5,1,1.5}{
            \draw (3,\x)--(0);
    }

\foreach \y in {1,2,3}{
            \draw (\y)--(0);
    }

\end{tikzpicture}
\end{center}
\caption{Sketch of the partial join $K_1 \vee (3K_1 \cup H)$, where the high degree vertex $v$ is coloured grey. The edges of $H$ are not drawn.}
\label{fig:star graph}
\end{figure}

To prove this, let $G=GS_{1,\ldots,1,|H|}$ be a generalized star with $k$ arm lengths equal to $1$ and one arm length equal to $|H|$, $G[V_1]=K_{1,k}$, $G[V_2]=P_{|H|}$, and let $v \in V(G)$ be the high degree vertex of $G$.  By \cite[Theorems~14 and 15]{MR2022294} any matrix in $S(G)$ has the unordered multiplicity list $\underline m$, such that $m_1,\ldots,m_\ell$ is majorized by $(k+1,1^{(|H|-1)})$, hence $\sum_{i\in [\ell]} m_i\leq k+\ell$. By Theorem~\ref{thm:partial join} there exists a matrix in $(K_{1,k},\{v\}) \vee H=K_1 \vee (k K_1 \cup H)$ with the same eigenvalues as a matrix in $S(G)$.
\end{example}

 \bibliographystyle{amsplain}
\bibliography{qqCref}

\end{document}